\documentclass[a4paper,11pt]{article}

\usepackage[utf8]{inputenc}
\usepackage[english]{babel}
\usepackage{amsmath,amssymb,amsthm} 
\usepackage{hyperref,xcolor,fullpage}
\usepackage{relsize}
\usepackage{dsfont}
\usepackage{enumitem}
\usepackage[numbers]{natbib} 

\allowdisplaybreaks 
\usepackage[numbers]{natbib} 

\newtheorem{theorem}{Theorem}[section]
\newtheorem{lemma}[theorem]{Lemma}
\newtheorem{definition}[theorem]{Definition}

\newtheorem{conjecture}[theorem]{Conjecture}

\newtheorem{proposition}[theorem]{Proposition}
\newtheorem{remark}[theorem]{Remark}

\makeatletter
\let\save@mathaccent\mathaccent
\newcommand*\if@single[3]{
	\setbox0\hbox{${\mathaccent"0362{#1}}^H$}%
	\setbox2\hbox{${\mathaccent"0362{\kern0pt#1}}^H$}%
	\ifdim\ht0=\ht2 #3\else #2\fi }
\newcommand*\rel@kern[1]{\kern#1\dimexpr\macc@kerna}
\newcommand*\widebar[1]{\@ifnextchar^{{\wide@bar{#1}{0}}}{\wide@bar{#1}{1}}}
\newcommand*\wide@bar[2]{\if@single{#1}{\wide@bar@{#1}{#2}{1}}{\wide@bar@{#1}{#2}{2}}}
\newcommand*\wide@bar@[3]{
	\begingroup
	\def\mathaccent##1##2{
		\let\mathaccent\save@mathaccent
		\if#32 \let\macc@nucleus\first@char \fi
		\setbox\z@\hbox{$\macc@style{\macc@nucleus}_{}$}
		\setbox\tw@\hbox{$\macc@style{\macc@nucleus}{}_{}$}
		\dimen@\wd\tw@ \advance\dimen@-\wd\z@ \divide\dimen@ 3 \@tempdima\wd\tw@ \advance\@tempdima-\scriptspace \divide\@tempdima 10 \advance\dimen@-\@tempdima \ifdim\dimen@>\z@ \dimen@0pt \fi \rel@kern{0.6}\kern-\dimen@
		\if#31 \overline{\rel@kern{-0.6}\kern\dimen@\macc@nucleus\rel@kern{0.4}\kern\dimen@} \advance\dimen@0.4\dimexpr\macc@kerna \let\final@kern#2 \ifdim\dimen@<\z@ \let\final@kern1 \fi
		\if \final@kern1 \kern-\dimen@ \fi
		\else \overline{\rel@kern{-0.6}\kern\dimen@#1} \fi }
	\macc@depth\@ne	\let\math@bgroup\@empty \let\math@egroup\macc@set@skewchar 	\mathsurround\z@ \frozen@everymath{\mathgroup\macc@group\relax} 	 \macc@set@skewchar\relax \let\mathaccentV\macc@nested@a	\if#31 \macc@nested@a\relax111{#1} \else \def\gobble@till@marker##1\endmarker{} \futurelet\first@char\gobble@till@marker#1\endmarker \ifcat\noexpand\first@char A\else \def\first@char{} \fi \macc@nested@a\relax111{\first@char} \fi
	\endgroup }
\makeatother

\newcommand{\bi}{\textbf{i}}
\newcommand{\bj}{\textbf{j}}
\newcommand{\bb}{\textbf{b}}
\newcommand{\bc}{\textbf{c}}

\newcommand{\bt}{\textbf{t}}

\usepackage{mathrsfs}
\DeclareMathAlphabet{\mathpzc}{OT1}{pzc}{m}{it}
\newcommand{\fp}[1][]{ \ifthenelse{\isempty{#1}}{\mathpzc{p}}{\mathpzc{p}(#1)} }

\newcommand{\sleq}{\scriptscriptstyle{\leq}}

\newcommand{\mD}{\mathcal{D}}

\newcommand{\mA}{\mathcal{A}}

\newcommand{\CC}{\mathbb{C}}
\newcommand{\RR}{\mathbb{R}}
\newcommand{\QQ}{\mathbb{Q}}
\newcommand{\NN}{\mathbb{N}}
\newcommand{\ZZ}{\mathbb{Z}}

\definecolor{darkred}{cmyk}{.3,.9,.80,.2}

\title{On a problem of S{\'a}rk\"ozy and S{\'o}s for multivariate linear forms}
\date{}
\author{
	Juanjo Ru{\'e}
\and
	Christoph Spiegel\thanks{Universitat Polit\`ecnica de Catalunya, Department of Mathematics, Edificio Omega, 08034 Barcelona, Spain, and Barcelona Graduate School of Mathematics. E-mail: {\tt juan.jose.rue@upc.edu}, {\tt christoph.spiegel@upc.edu}. Supported by the Spanish Ministerio de Econom\'{i}a y Competitividad projects MTM2014-54745-P, MTM2017-82166-P and the María de Maetzu research grant MDM-2014-0445. C. S. is also suppported by an FPI grant under the project projects MTM2014-54745-P.}
}

\begin{document}
\maketitle

\begin{abstract}
We prove that for pairwise co-prime numbers $k_1,\dots,k_d \geq 2$ there does not exist any infinite set of positive integers $\mA$ such that the representation function $r_{\mA}(n) = \# \big\{ (a_1, \dots, a_d) \in \mA^d : k_1 a_1 + \dots + k_d a_d = n \big\}$ becomes constant for $n$ large enough. This result is a particular case of our main theorem, which poses a further step towards answering a question of S{\'a}rk\"ozy and S{\'o}s and widely extends a previous result of Cilleruelo and Rué for bivariate linear forms (Bull. of the London Math. Society 2009).
\end{abstract}

\begin{center}
\emph{Javier Cilleruelo, in Memoriam}
\end{center}

\section{Introduction}

Let $\mA \subseteq \NN_0$ be an infinite set of positive integers and $k_1, \dots, k_d \in \NN$.
We are interested in studying the behaviour of the representation function
\begin{equation*}
	r_{\mA}(n) = r_{\mA}(n;k_1,\dots,k_d) = \# \big\{ (a_1, \dots, a_d) \in \mA^d : k_1 a_1 + \dots + k_d a_d = n \big\}.
\end{equation*}
More specifically, S{\'a}rk\"ozy and S{\'o}s~\cite[Problem 7.1.]{SarkozySos_1997} asked for which values of $k_1,\dots,k_d$ one can find an infinite set $\mA$ such that the function $r_{\mA}(n;k_1,\dots,k_d)$ becomes constant for $n$ large enough. For the base case, it is clear that $r_{\mA}(n;1,1)$ is odd whenever $n = 2a$ for some $a \in \mA$ and even otherwise, so that the representation function cannot become constant. For $k \geq 2$, Moser~\cite{Moser_1962} constructed a set $\mA$ such that $r_{\mA}(n;1,k) = 1$ for all $n \in \NN_0$. The study of bivariate linear forms was completely settled by Cilleruelo and the first author~\cite{CillerueloRue_2009} by showing that the only cases in which $r_{\mA}(n;k_1,k_2)$ may become constant are those considered by Moser.

The multivariate case is less well studied. If $\gcd(k_1,\dots,k_d) > 1$, then one trivially observes that $r(n;k_1,\dots,k_d)$ cannot become constant.
The only non-trivial case studied so far was the following: for $m > 1$ dividing $d$, the first author~\cite{Rue_2011} showed that if in the $d$--tuple of coefficients $(k_1,\dots,k_d)$ each element is repeated $m$ times, then there cannot exists an infinite set $\mA$ such that $r_{\mA}(n;k_1,\dots,k_d)$ becomes constant for $n$ large enough. This for example covers the case $(k_1,k_2,k_3,k_4,k_5,k_6) =(2,4,6,2,4,6)$. Observe that each coefficient in this example is repeated twice, that is $m = 2$.

In this paper we provide a step beyond this result and show that whenever the set of coefficients is pairwise co-prime, then there does not exists any infinite set $\mA$ for which $r(n;k_1,\dots,k_d)$ is constant for $n$ large enough. This is a particular case of our main theorem, which covers a wide extension of this situation:

\begin{theorem} \label{thm:main_result}
	Let $q_1,\dots,q_m \geq 2$ be pairwise co-prime integers and $b(i,j) \in \{0,1\}$, so that for each $1 \leq i \leq d$ there exists some $1 \leq j \leq m$ such that $b(i,j)=1$. If $k_i = q_1^{b(i,1)} \cdots q_m^{b(i,m)}$ for $1 \leq i \leq d$, then for every infinite set $\mA \subseteq \NN_0$ the function $r_{\mA}(n;k_1,\dots,k_d)$ cannot become constant.
\end{theorem}

In particular, if $m = d$ and $b(i,j)=1$ if and only if $i=j$, then this represents the case where $k_1,\dots,k_d\geq 2$ are pairwise co-prime numbers. Other new cases covered by this result are for instance $(k_1,k_2,k_3)=(2,3,2\times 3)$ as well as $(k_1,k_2,k_3,k_4)=(2^2\times 3, 2^2\times 5, 3\times 5, 2^2\times 3\times 5)$.

Our method starts with some ideas introduced in \cite{CillerueloRue_2009} dealing with generating functions and cyclotomic polyomials (see Section \ref{sec:prel}). The main new idea in this paper is to use an inductive argument in order to be able to show that a certain multivariate recurrence relation is not possible to be satisfied unless some initial condition is trivial.


\section{Preliminaries}\label{sec:prel}

\paragraph{\textbf{Generating functions}.} The language in which we will approach this problem goes back to~\cite{Dirac_1951}. Let $f_{\mA}(z) = \sum_{a \in \mA} z^{a}$ denote the \emph{generating function} associated with $\mA$. By a simple argument over the generating functions, it is easy to verify that the existence of a set $\mA$ for which $r_{\mA}(n;k_1,\dots,k_d)$ becomes constant would imply that
\begin{equation*} \label{eq:rephrasedasgeneratingfunctions}
	f_{\mA}(z^{k_1}) \cdots f_{\mA}(z^{k_d}) = \frac{P(z)}{1-z}
\end{equation*}
for some polynomial $P$ with positive integer coefficients satisfying $P(1) \neq 0$. To simplify notation, we will generally consider the $d$--th power of this equation, that is for $F_{\mA}(z) = f_{\mA}^d(z)$ we have
\begin{equation} \label{eq:rephrasedasgeneratingfunctions2}
	F_{\mA}(z^{k_1}) \cdots F_{\mA}(z^{k_d}) = \frac{P^d(z)}{(1-z)^d}.
\end{equation}
Let us remark two obvious but important properties of $f_{\mA}$ and $F_{\mA}$.
\begin{remark} \label{rmk:coefficients}
	$f_{\mA}(z)$ is a formal power series with coefficients in $\{0,1\}$ that therefore is analytic in the open complex disc $\mD = \{z \in \CC : |z| < 1 \}$. It follows that $F_{\mA}(z)$ is also a formal power series with positive coefficients that is analytic in $\mD$.
\end{remark}
This is the starting point of the proof of Theorem~\ref{thm:main_result}, mainly building upon the tools developed in~\cite{CillerueloRue_2009} using properties relating to cyclotomic polynomials that we will now briefly review.

\paragraph{\textbf{Cyclotomic polynomials}.} The \emph{cyclotomic polynomial of order $n$} is defined as
\begin{equation*}
	\Phi_n(z) = \prod_{\xi \in \phi_n} (z - \xi) \in \ZZ[z]
\end{equation*}
where
\begin{equation}
	\phi_n = \big\{ e^{\frac{2\pi i \ell}{n}} : 0 \leq \ell < n \text{ satisfying } (\ell,n) = 1 \big\} = \big\{ \xi \in \CC : \xi^k = 1 \text{ iff } k \equiv 0 \mod n \big\}
\end{equation}
denotes the \emph{set of primitive roots of unity of order $n \in \NN$}. It is well known that $\Phi_n(z) \in \ZZ[z]$, that is it has integer coefficients. Cyclotomic polynomials have the property of being irreducible over $\ZZ[z]$ and therefore it follows that for any polynomial $P(z) \in \ZZ[z]$ and $n \in \NN$ there exists an integer $s_n \in \NN_0$ such that
\begin{equation} \label{eq:factorP}
	P_n(z) := P(z) \, \Phi_n^{-s_{n}}(z)
\end{equation}
is a polynomial in $\ZZ[z]$ satisfying $P_n (\xi) \neq 0$ for all $\xi \in \phi_n$. We will say that we have \emph{factored $\Phi_n(z)$ out of $P(z)$ with multiplicity $s_n$}. Note that the multiplicity is trivially unique.

This is not guaranteed to be possible for arbitrary non-polynomial functions. In particular, our function $F_{\mA}(z)$ is not even analytic at roots of unity and it can also be shown that even the radial limit of $F_{\mA}(z)$, where $z$ approaches some root of unit $\xi$ radially from within $\mD$, may not exist in general. However, we can extend our notion of factoring out cyclotomic polynomials in a natural way that will be applicable to our function $F_{\mA}(z)$.
\begin{definition}
	Let $n \in \NN$ and $F(z)$ some function analytic in $\mD$. We say that we can \emph{factor $\Phi_n(z)$ out of $F(z)$ with multiplicity $r_n$} if, for any $\xi \in \phi_n$ and sequence $\{z_k : k \in \NN \} \subset [0,1)$ converging to $1$, the limit of $|F(z_k \, \xi) \, \Phi_n^{-r_n}(z_k \, \xi)|$ as $k \to \infty$ either does not exist or is not equal to $0$ or $\infty$.
\end{definition}
Note that $|F(z_k \, \xi) \, \Phi_n^{-r_n}(z_k \, \xi)|$ not going to infinity is the same as $F(z_k \, \xi) \, \Phi_n^{-r_n}(z_k \, \xi)$ being bounded. We remark that, by continuity, this notion is a true extension of the previous one for polynomials. It is also again easy to verify that the multiplicity, if it exists, is uniquely determined.
\begin{lemma} \label{lemma:uniqueness}
	If we can factor $\Phi_n(z)$ out of $F(z)$, then the multiplicity is uniquely determined.
\end{lemma}
\begin{proof}
	Assume that we can factor $\Phi_n(z)$ out of $F(z)$ with multiplicity $r_n$. Let $\{z_k : k \in \NN \} \subset [0,1)$ be a sequence  converging to $1$ and $\xi \in \phi_n$. Consider
	\begin{equation} \label{eq:duh}
		|F(z_k \, \xi) \, \Phi_n^{-r_n + \alpha}(z_k \, \xi)| = |F(z_k \, \xi) \, \Phi_n^{-r_n}(z_k \, \xi)| \, |\Phi_n^{\alpha}(z_k \, \xi)|	
	\end{equation}
	as $k$ goes to infinity. As $\Phi_n(\xi) = 0$ and $F|(z_k \, \xi) \, \Phi_n^{-r_n}(z_k \, \xi)|$ is bounded and does not go to $0$, \eqref{eq:duh} must tend to $0$ if $\alpha > 0$ and to $\infty$ if $\alpha < 0$. It follows that the multiplicity must be uniquely determined.
\end{proof}

Let us introduce some short-hand notation for this that we will use in the next section. If $q_1,\dots, q_m$ are fixed co-prime integers as given by Theorem~\ref{thm:main_result} and $\bj = (j_1,\dots,j_m) \in \NN_0^m$, then we write
\begin{equation*}
	\Phi_{\bj}(z) := \Phi_{q_1^{j_1} \cdots q_m^{j_m}}(z), \enspace
	\phi_{\bj} := \phi_{q_1^{j_1} \cdots q_m^{j_m}}, \enspace
	s_{\bj} := s_{q_1^{j_1} \cdots q_m^{j_m}} \enspace \text{and} \enspace
	r_{\bj} := r_{q_1^{j_1} \cdots q_m^{j_m}}.
\end{equation*}

The main strategy of the proof is to show that for any $\bj \in \NN_0^m$ we can factor $\Phi_{\bj}(z)$ out of our hypothetical function $F_{\mA}(z) = f_{\mA}^d(z)$ satisfying \eqref{eq:rephrasedasgeneratingfunctions2} and that the multiplicites $r_{\bj}$ have to fulfil certain relations between themselves. The goal will be to find a contradiction in these relations, negating the possibility of such a function and therefore such a set $\mA$ existing in the first place. Before formally establishing these relations in the next section, let us introduce two lemmata that we will need.

\begin{lemma} \label{lemma:cyclotomicbasics}
	Given $k,n \in \NN$ such that $k \mid n$ we have $\phi_{n/k} = \{ \xi^k : \xi \in \phi_n \}$. Furthermore, we can factor $\Phi_{n}(z)$ out of $\Phi_{n/k}(z^k)$ with multiplicity $1$.
\end{lemma}

\begin{proof}
	To see equality between the two sets, observe that 
	\begin{align*}
		\big\{ \xi^k : \xi \in \phi_n \big\} & = \big\{ \xi^k : \xi^{\ell} = 1 \text{ iff } \ell \equiv 0 \mod n \big\} \\
		& = \big\{ \xi^k : (\xi^k)^{\ell/k} = 1 \text{ iff } \ell \equiv 0 \mod n \big\} \\
		& = \big\{ \xi^k : (\xi^k)^{\ell} = 1 \text{ iff } \ell \equiv 0 \mod n/k \big\} = \phi_{n/k}.
	\end{align*}
	As $\Phi_{n/k}(z^{k})$ is a polynomial in $\ZZ[z]$ and $\Phi_{n/k} (\xi^{k}) = 0$ for any $\xi \in \phi_n$ via the previous observation, it follows that we can factor out $\Phi_n(z)$. The multiplicity is equal to $1$ since all roots of $\Phi_{n/k}(z^{k})$ are simple.
\end{proof}

Lastly, we will also need the following technical lemma that will allow us to draw conclusions from the limit of certain types of products to the limits of its individual factors.

\begin{lemma} \label{lemma:limit}
	Let $F(z) = \sum_{n \in \NN} a_n z^n$ be a formal power series with positive coefficients that is analytic in $\mD$. If there exists a sequence $\{z_k : k \in \NN\} \subset [0,1)$ tending to $1$ such that $|F(z_k) (1-z_k)|$ goes to $0$, then so does $|F(z_k^\alpha) (1-z_k^{\alpha})|$ for any $\alpha \neq 0$. If $|F(z_k)(1-z_k)|$ goes to $\infty$, then so does $|F(y_k^{\alpha})(1-y_k^{\alpha})|$ for some subsequence $\{y_k : k\in \NN\} \subseteq \{z_k : k \in \NN\}$ also tending to $1$.
\end{lemma}
\begin{proof}
	We start by observing that, as $F(z)$ has positive coefficients, we may omit the absolute values since $|F(z)(1-z)| = F(z)(1-z)$ for any $0 \leq z < 1$. Let us start with the first case, that is $F(z_k)(1-z_k)$ going to $0$. We have
	\begin{align*}
		\big|F(z_k)(1-z_k) - F(z_k^{\alpha})(1-z_k^{\alpha}) \big| & = \bigg| \sum_n a_n \, z_k^n \, (1-z_k) \: - \:  \sum_n a_n \, z_k^{\alpha n} \, (1-z_k^{\alpha}) \bigg| \\
		& = \bigg| \sum_n a_n \, z_k^n \, \big( (1-z_k) - z_k^{(\alpha-1) n} (1-z_k^{\alpha}) \big) \bigg| \\
		& \leq \sum_n \big|a_n \, z_k^n \big| \big|  (1-z_k) - z_k^{(\alpha-1) n} (1-z_k^{\alpha}) \big| \\
		&\leq \sum_n \big| a_n \, z_k^n \big| = F(z_k) \to 0.
	\end{align*}
	In the last equality we have used the fact that the coefficients $a_n$ are positive and that $z_k \in [0,1)$ so that $|a_n z_k^n| = a_n z_k^n$. It clearly follows that $F(z_k^{\alpha})(1-z_k^{\alpha})$ must go to $0$ as well.
	
	Next, assume that $F(z_k) (1-z_k)$ goes to $\infty$ but that $F(z_k^{\alpha}) (1-z_k^{\alpha})$ is bounded, that is $F(z_k^{\alpha}) (1-z_k^{\alpha}) \in [-M,M]$ for some $M \in \RR$ and any $k \in \NN$. It follows that
	\begin{align*}
		\big|F(z_k)(1-z_k) - F(z_k^{\alpha})(1-z_k^{\alpha}) \big| & = \bigg| \sum_n a_n \, z_k^{\alpha n} \, \big( z_k^{(1-\alpha) n} (1-z_k) - (1-z_k^{\alpha}) \big) \bigg| \\
		& \leq \sum_n \big|a_n \, z_k^{\alpha n} \big| \big| z_k^{(1-\alpha) n} (1-z_k) - (1-z_k^{\alpha}) \big| \\
		&\leq \sum_n \big| a_n \, z_k^{\alpha n} \big| = F(z_k^{\alpha}) \leq M,
	\end{align*}
	a contradiction since $|F(z_k)(1-z_k) - F(z_k^{\alpha})(1-z_k^{\alpha})| \to \infty$. Since $F(z_k^{\alpha}) (1-z_k^{\alpha})$ must therefore be unbounded, there exists some subsequence $\{y_k : k \in \NN\} \subset \{z_k : k \in \NN \} \subset [0,1)$ such that $F(y_k^{\alpha}) (1-y_k^{\alpha})$ goes to $\infty$.
\end{proof}

\section{Recurrence relations}\label{sec:rec}

We can now give the statement and proof establishing that we can factor any $\Phi_{\bj}(z)$ out of our function $F_{\mA}(z)$ and that the multiplicities satisfy certain relations. We will in fact state this for any $k_1,\dots,k_d \in \NN$ and later derive a contradiction from these relations in the specific case stated in Theorem~\ref{thm:main_result}.

For any $a,b \in \NN_0$, $\bj = (j_1,\dots,j_m) \in \NN_0^m$ and $\bb = (b_1,\dots,b_m) \in \NN_0^m$, we will use the notation
	\begin{equation*}
		a \ominus b = \max\{a-b,0\} \quad \text{and} \quad \bj \ominus \bb = ( j_1 \ominus b_1, \dots, j_m \ominus b_m ).
	\end{equation*}
Let us state the central proposition of this section.

\begin{proposition} \label{prop:FundamentalRelation}
	Let $q_1,\dots,q_m \geq 2$ be pairwise co-prime integers and $k_i = q_1^{b(i,1)} \cdots q_m^{b(i,m)}$ for $1 \leq i \leq d$ where $b(i,j) \in \NN_0$. Furthermore, let $P(z) \in \ZZ[z]$ be a polynomial satisfying $P(1) \neq 0$ and $F(z)$ a formal power series with positive coefficients that is analytic in $\mD$ such that
	\begin{equation} \label{eq:rephrasedasgeneratingfunctions3}
		F(z^{k_1}) \cdots F(z^{k_d}) = \frac{P^d(z)}{(1-z)^d}.
	\end{equation}
	Then for all $\bj \in \NN_0^m$ there exist integers $r_{\bj} \in \NN_0$ so that we can factor $\Phi_{\bj}$ out of $F$ with multiplicity $r_{\bj}$. Writing $\bb_i = (b(i,1),\dots,b(i,m))$ for $1 \leq i \leq m$ as well as $s_{\bj} \in \NN_0$ for the integer satisfying $P(\xi) \, \Phi_{\bj}^{-s_{\bj}}(\xi) \neq 0$ for any $\xi \in \phi_{\bj}$, these multiplicities satisfy the relations
	\begin{equation} \label{eq:FundamentalRelation}
		r_{\bold{0}} = -1 \quad \text{and} \quad r_{\bj \, \ominus \bb_1} + \dots + r_{\bj \, \ominus \bb_d} = d s_{\bj} \quad\mathrm{for\,\,all\,\,} \bj \in \NN_0^m \setminus \{ \bold{0} \}
	\end{equation}
	and we have $r_{\bi} \equiv -1 \mod d$ for all $\bi \in \NN_0^m$.
\end{proposition}

\begin{proof}
We start assuming that the set of multiplicities $\{r_{\bj} : \bj \in \NN_0^m\}$ exists and show that the relations given by \eqref{eq:FundamentalRelation} must be satisfied. After this, we will show that there is a way to recursively determine the values $\{r_{\bj} : \bj \in \NN_0^m\}$, proving their existence.

Let us start with $r_{\bold{0}} = -1$. For $F_{\bold{0}}(z) := F(z) (1-z)$ we wish to show that there does not exist any sequence $\{z_k : k \in \NN \} \subset [0,1)$ going to $1$ such that $|F_{\bold{0}}(z_k)|$ tends to either $0$ or $\infty$. Note that if such a sequence were to exist, then by iteratively applying Lemma~\ref{lemma:limit} we would obtain some subsequence $\{y_k : k \in \NN\} \subseteq \{z_k : k \in \NN\}$ still tending to $1$ such that all $|F_{\bold{0}}(y_k^{k_\ell})|$ would collectively tend to either $0$ or $\infty$ for any $1 \leq \ell \leq d$.

Inserting the equality $F(z) =  (1-z)^{-1} F_{\bold{0}}(z)$ into \eqref{eq:rephrasedasgeneratingfunctions3} and observing that $(1-z)/(1-z^{k_{\ell}}) = (1 + z + \dots + z^{k_{\ell}-1})^{-1}$, we get that $F_{\bold{0}}(z)$ satisfies
	\begin{equation*}
		\prod_{\ell = 1}^d |(1 + z + \dots + z^{k_{\ell}-1})^{-1}| \, | F_{\bold{0}}(z^{k_{\ell}})| = |P^d (z)|.	
	\end{equation*}
	As $P^d(1) \neq 0$ as well as $(1 + 1 + \dots + 1^{k_{\ell}-1})^{-1} = 1/k_{\ell} \neq 0$ for $1 \leq \ell \leq d$ it follows that there cannot exist a sequence $\{z_k : k \in \NN \} \subset [0,1)$ tending to $1$ such that $|F_{\bold{0}}(z^{k_{\ell}})|$ all collectively tend to $0$ or $\infty$ for any $1 \leq \ell \leq d$, proving the desired statement.
	
	Next, let us show that if for a given $\bj \in \NN_0^m \setminus \{ \bold{0} \}$ the values $r_{\bj \ominus \bb_1}$, \dots, $r_{\bj \ominus \bb_d}$ exist, then they must satisfy the relation given by \eqref{eq:FundamentalRelation}. For $1 \leq i \leq d$ let
	\begin{equation*}
		F_{\bj \ominus \bb_i} := F(z) \, \Phi_{\bj \ominus \bb_i}^{-r_{\bj \ominus \bb_i}}
	\end{equation*}
	and rewrite \eqref{eq:rephrasedasgeneratingfunctions3} as
	\begin{equation} \label{eq:recursion1}
		\Phi_{\bj \ominus \bb_1}^{r_{\bj \ominus \bb_1}} (z^{k_1}) \, F_{\bj \ominus \bb_1}(z^{k_1}) \: \cdots \: \Phi_{\bj \ominus \bb_d}^{r_{\bj \ominus \bb_d}} (z^{k_d}) \, F_{\bj \ominus \bb_d}(z^{k_d}) = \frac{\Phi_{\bj}^{ds_{\bj}}(z) \, P_{\bj}^d(z)}{(1-z)^d}.
	\end{equation}
	Writing $R_{\bj,i}(z) := \Phi_{\bj \ominus \bb_i}(z^{k_i}) \, \Phi_{\bj}^{-1}(z)$ we can restate \eqref{eq:recursion1} as
	\begin{equation} \label{eq:recursion2}
		\Phi_{\bj}^{r_{\bj \ominus \bb_1} + \dots + r_{\bj \ominus \bb_d} - ds_{\bj}} (z) \: \Big( R_{\bj,1}^{r_{\bj \ominus \bb_1}}(z) \, F_{\bj \ominus \bb_1}(z^{k_1}) \: \cdots \: R_{\bj,d}^{r_{\bj \ominus \bb_d}}(z) \, F_{\bj \ominus \bb_d}(z^{k_d}) \Big) = \frac{P_{\bj}^d(z)}{(1-z)^d}.
	\end{equation}
	We observe that, by assumption as well as Lemma~\ref{lemma:cyclotomicbasics}, if we substitute $z_k \, \xi$ into \eqref{eq:recursion2} where $\xi \in \phi_{\bj}$ and $\{z_k : k \in \NN\} \subset [0,1)$ tends to $1$ and take absolute values, then all involved factors but the first one converge neither to $0$ nor to $\pm \infty$. As $\Phi_{\bj}^{\alpha} (z_k \, \xi)$ tends to either $0$ or $\infty$ for any $\alpha \neq 0$, it follows that the desired relation must hold.

	\medskip

	It remains to be shown that the values $r_{\bj}$ actually exist for any $\bj \in \NN_0$. We will do so recursively with the base case of $r_{\bf 0} = -1$ already having been established. From now on, let us -- for simplicities sake -- redefine the value $s_{\bf 0}$ (which previously was $0$ as $P(0) \neq 0$) to be $s_{\bf 0} = -1$, so that the initial relation $r_{\bf 0} = -1$ is now included in the general relation for the case $\bj = {\bf 0}$. We observe that if for some $1 \leq \ell \leq d$ all values $r_{\bj \ominus \bb_1}, \dots, r_{\bj \ominus \bb_d}$ except for $r_{\bj \ominus \bb_{\ell}}$ have already been shown to exist, then through the already established \eqref{eq:recursion2} it is clear that setting
	\begin{equation*}
		r_{\bj \ominus \bb_{\ell}} = ds_{\bj} - \sum_{i \neq \ell} r_{\bj \ominus \bb_i}
	\end{equation*}
	would give the desired property, that is for no sequence $\{z_k : k \in \NN \} \subset [0,1)$ going to $1$ and $\xi \in \phi_{\bj \ominus \bb_{\ell}}$ could $F(z_k \, \xi) \, \Phi_{\bj \ominus \bb_{\ell}}^{-r_{\bj \ominus \bb_{\ell}}}(z_k \, \xi)$ go to either $0$ or $\pm \infty$. We therefore wish to show inductively that for all $\bi \in \NN_0^m$ there exists a $\bj \in \NN_0^m$ and $1 \leq \ell \leq d$ such that $\bi = \bj \ominus \bb_{\ell}$ and all other involved values $\bj \ominus \bb_1,\dots,\bj \ominus \bb_{\ell-1},\bj \ominus \bb_{\ell+1},\dots,\bj \ominus \bb_d$ have already been determined by the inductive  hypothesis.

	For this we will give the indices $\bj \in \NN_0^m$ inducing these relations an appropriate ordering. More preciesly, for each $\bj = (j_1,\dots,j_m) \in \NN_0^m$ let $\bj^{\sleq} = (j^{\sleq}_1, \dots, j^{\sleq}_m)$ denote the \emph{ordered version}, that is $j^{\sleq}_1 \leq j^{\sleq}_2 \leq \dots \leq j^{\sleq}_m$ and there exists some permutation $\sigma$ on $m$ letters such that $\bj = \big(j^{\leq}_{\sigma(1)}, \dots, j^{\leq}_{\sigma(m)}\big)$. Consider the ordering on $\NN_0^m$ given by $\bj \prec \bj'$ if $\bj^{\sleq}$ lexicographically comes before $\bj'^{\sleq}$. In this situation, ties are broken arbitrarily. We want to show that going through the indices $\bj$ in that order and considering the relation $r_{\bj \, \ominus \bb_1} + \dots + r_{\bj \, \ominus \bb_d} = d s_{\bj}$, then at most one of the $r_{\bj \, \ominus \bb_{\ell}}$ will not have occurred in any of the previous relations given by some $\bj' \prec \bj$.
		
	Assume to the contrary that there exist $\bi \neq \bi' \in \NN_0^{m}$ such that, for both of them, $\bj \in \NN_0^m$ is the first index for which there exist $1 \leq \ell,\, \ell' \leq d$ satisfying $\bi = \bj \ominus \bb_{\ell}$ and $\bi' = \bj \ominus \bb_{\ell'}$. Note that $\bb_{\ell} \neq \bb_{\ell'}$ and therefore at least one of the two statements $\bj \ominus (\bb_{\ell} - \bb_{\ell'}) \prec \bj$ and $\bj \ominus (\bb_{\ell'} - \bb_{\ell}) \prec \bj$ must hold. To see this, assume without loss of generality that $\bj = (j_1,\dots,j_m)$ is already in ordered form. Note that $\bb_{\ell} - \bb_{\ell'} \neq {\bf 0}$ as $\bi \neq \bi'$. Writing $\bb_{\ell} = (b_1,\dots,b_m)$ and $\bb_{\ell'} = (b_1',\dots,b_m')$, and letting $1 \leq i \leq m$ be the first index such that $b_i \neq b_{i}'$ and $j_i > 0$, then we clearly have that either
	\begin{equation*}
		j_i \ominus (b_i - b_{i'}) = \max\{j_i - (b_i - b_{i'}),0\} < j_i \quad \text{or} \quad j_i \ominus (b_{i'} - b_i) = \max\{j_i + (b_i - b_{i'}),0\} < j_i,
	\end{equation*}
	meaning that at least one of the two values $\bj \ominus (\bb_{\ell} - \bb_{\ell'})$ and $\bj \ominus (\bb_{\ell'} - \bb_{\ell})$ must lexicographically come before $\bj$. Note that such index $i$ must exist since if $j_i = 0$ whenever $b_i - b_i' \neq 0$ then we would have had $\bi = \bj \ominus \bb_{\ell} = \bb_{\ell'} = \bi'$ in contradiction to our assumption that $\bi \neq \bi'$.
	
	Assume now without loss of generality that $\bj \ominus (\bb_{\ell} - \bb_{\ell'}) \prec \bj$. Since for $a,b,c \geq 0$ we trivially have that $\max\{\max\{a-b+c,0\}-c,0\} = \max\{\max\{a-b,-c\},0\} = \max\{a-b,0\}$, it follows that
	\begin{equation*}
		\big( \bj \ominus (\bb_{\ell} - \bb_{\ell'}) \big) \ominus \bb_{\ell'} = \bj \ominus \bb_{\ell} = \bi.
	\end{equation*}
	This is however in contradiction to the requirement that $\bj$ was the smallest index with respect to the ordering $\prec$ for which the relation given by \eqref{eq:FundamentalRelation} involves $r_{\bi}$, giving us the desired result.

	Finally, note that from the previous argument it also inductively follows that $r_{\bi} \equiv -1 \mod d$ for all $\bi \in \NN_0^m$ as in the base case we have that $r_{\bold 0} = -1$.
\end{proof}

\section{Proof of Theorem~\ref{thm:main_result}}

We will now use the proposition established in the previous section to prove Theorem~\ref{thm:main_result} by contradiction. We start by introducing some necessary notation and definitions. We write $\bc_i = (c(i,1),\dots,c(i,m)) \in \NN_0^m$ and for any $1 \leq \ell \leq m$  we use the notation
			\begin{equation*}
				S_{\ell} = \{ 1 \leq i \leq d : c(i,\ell) = 0 \} \quad \text{and} \quad S_{\ell}' = \{ 1, \dots, d \} \setminus S_{\ell}.
			\end{equation*}
We will also use the following notation: for any $\bi = (i_1,\dots,i_{m-1}) \in \NN_0^{m-1}$ and $1 \leq \ell \leq m$ let
$$\Delta_{\bi,\ell} = v_{(i_1,\dots,i_{\ell-1},1,i_{\ell},\dots,i_{m-1})} - v_{(i_1,\dots,i_{\ell-1},0,i_{\ell},\dots,i_{m-1})}.$$
Finally, for $1\leq \ell \leq m$, we write $\mathds{1}_{\ell}\in \NN_0^{m}$ for the vector whose entries are all equal to 0 except for the $\ell$--th entry, which is equal to 1.

\begin{definition}For $m \geq 1$, we define an \emph{$m$--structure} to be any set of values $\{ v_{\bj} \in \QQ : \bj \in \NN_0^m \}$	for which there exist $\bc_1, \dots, \bc_d \in \NN_0^m$ and $\{ u_{\bj} \in \ZZ : \bj \in \NN_0^m \setminus \{\bf 0\} \}$ so that the values satisfy the relation
			\begin{equation*}
				v_{\bj \, \ominus \bc_1} + \dots + v_{\bj \, \ominus \bc_d} = u_{\bj} \quad \text{for all } \bj \in \NN_0^m \setminus \{\bf 0\}.
			\end{equation*}
Additionally, we define the following:
	\begin{enumerate}
		\item	We say that an $m$--structure is \emph{regular} if we have that the corresponding vectors $\bc_1, \dots, \bc_d \in \{0,1\}^m \setminus \{ \bold 0 \}$ for all $1 \leq i \leq d$ as well as $S_{\ell} \neq \emptyset$ for all $1 \leq \ell \leq m$.
		\item	We say that an $m$--structure is \emph{homogeneous outside $\bt = (t_1,\dots,t_m) \in \NN_0^m$} if the corresponding vectors $\{ u_{\bj} \in \ZZ : \bj \in \NN_0^m \setminus \{\bf 0\}\}$ satisfy $u_{\bj} = 0$ for all $\bj \in \NN_0^m \setminus [0,t_1] \times \dots \times [0,t_m]$.
	\end{enumerate}
\end{definition}

The first lemma shows a key ingredient in the inductive step developed later by reducing the value of $m$.

\begin{lemma} \label{lemma:differencestructure}
	For any $m$--structure $\{ v_{\bj} \in \QQ : \bj \in \NN_0^m\}$ that is homogeneous outside $\bt = (t_1,\dots,t_m) \in \NN_0^m$ and for which there exists $1 \leq \ell \leq m$ such that $|S_{\ell}| \neq 0$, the values $\left\{ \Delta_{\bi,\ell} : \bi \in \NN_0^{m-1}\right\}$ define an $(m-1)$--structure that is homogeneous outside $\bt_{\ell} = (t_1,\dots,t_{\ell-1},t_{\ell+1},\dots,t_m)$.
\end{lemma}

\begin{proof}
To simplify notation, assume without loss of generality that $\ell = m$. Let $\bc_1, \dots, \bc_d \in \NN_0^m$ and $\{ u_{\bj} \in \ZZ : \bj \in \NN_0^m \setminus \{\bf 0\}\}$ the corresponding sets of vectors given by the definition of $m$--structures.

For $i \in S_{\ell}$, let $\bc_i' = (c(i,1),\dots,c(i,m-1))$. Furthermore for $\bj' = (j_1,\dots,j_{m-1}) \in \NN_0^{m-1}$ let $\bj = (j_1,\dots,j_{m-1},0)$ and let $u_{\bj'} = u_{\bj + \mathds{1}_{\ell}} - u_{\bj}$. Using this notation, we have
	\begin{align*}
		\sum_{i \in S_{\ell}} \Delta_{\bj' \ominus \bc_i',\ell} & = \sum_{i \in S_{\ell}} v_{(\bj + \mathds{1}_{\ell}) \ominus \bc_i} - \sum_{i \in S_{\ell}} v_{\bj \ominus \bc_i} \\
		& = \Big( u_{\bj + \mathds{1}_{\ell}} - \sum_{i \in S_{\ell}'} v_{(\bj + \mathds{1}_{\ell}) \ominus \bc_i} \Big) - \Big( u_{\bj} - \sum_{i \in S_{\ell}'} v_{\bj \ominus \bc_i} \Big) = u_{\bj + \mathds{1}_{\ell}} - u_{\bj} = u_{\bj'}.
	\end{align*}
	Here we have used the fact that for $i \in S_{\ell}'$ we have $(\bj + \mathds{1}_{\ell}) \ominus \bc_i = \bj \ominus \bc_i$ as $c(i,\ell) \neq 0$. It follows that the values $\left\{ \Delta_{\bi,\ell} : \bi \in \NN_0^{m-1} \right\}$ form an $(m-1)$--structure with $\{ \bc_i' : i \in S_{\ell}\}$ and $\{ u_{\bj'} : \bj' \in \NN_0^{m-1} \setminus \{\bf 0\} \}$. As $u_{\bj'} = u_{\bj + \mathds{1}_{\ell}} - u_{\bj} = 0$ for $\bj' \in \NN_0^{m-1} \setminus [0,t_1] \times \dots \times [0,t_{m-1}]$, it follows that the structure is homogeneous outside $\bt_{\ell}$.
\end{proof}

\begin{lemma} \label{lemma:all0}
	A regular $m$--structure $\{ v_{\bj} \in \QQ : \bj \in \NN_0^m \}$ that is homogeneous outside $\bt = (t_1,\dots,t_m) \in \NN_0^m$ satisfies $v_{\bi} = 0$ for all $\bi \in \NN_0^m \setminus [0,t_1] \times \dots \times [0,t_m]$.
\end{lemma}

\begin{proof}
	We will prove the statement by induction on $m$. Let us start by showing the statement for $m = 1$. In this case, $\bc_1, \dots, \bc_d$ are non-zero, positive integers satisfying $\bc_1 = \dots = \bc_d = 1$ as the structure is regular. It follows that the relations defining the structure are of the type $d \, v_{\bj \ominus 1} = u_{\bj}$ for all $\bj \in \NN$. Since $u_{\bj} = 0$ for $\bj > \bt = t_1$, we have $v_{\bi} = 0$ for all $\bi \in \NN_0 \setminus [0,t_1 \ominus \bc_1] \subseteq \NN_0 \setminus [0,t_1]$ as desired.
	
	Now assume that the statement is true for all $(m-1)$--structures and let us show that then it must also hold for any $m$--structure. As the structure is regular, we have $S_{\ell} \neq \emptyset$ for all $1 \leq \ell \leq m$ and Lemma~\ref{lemma:differencestructure} shows that $\{ \Delta_{\bi,\ell} : \bi \in \NN_0^{m-1} \}$ is an $(m-1)$--structure that is homogeneous outside $\bt_{\ell}$ for any $1 \leq \ell \leq m$. Let us without loss of generality assume that $\ell = m$ to simplify notation. By the inductive assumption it follows that $\Delta_{\bi,\ell} = 0$ for all $\bi \in \NN_0^{m-1} \setminus [0,t_1] \times \dots \times [0,t_{m-1}]$. It follows that $\{ v_{\bi}' = v_{\bi + \mathds{1}_{\ell}} : \bi \in \NN_0^m \}$ is an $m$--structure where the corresponding $\{ u_{\bj}' : \bj \in \NN_0^m\}$ satisfying
	\begin{equation}
		u_{\bj}' = \begin{cases}
			u_{\bj + \mathds{1}_{\ell}}  & \text{for } \bj = (j_1,\dots,j_m) \text{ s.t. } j_{\ell} \neq 0, \\
			u_{\bj + \mathds{1}_{\ell}} + \sum_{i \in S_{\ell}'} \Delta_{\bj \ominus \bc_i,\ell} & \text{for } \bj = (j_1,\dots,j_m) \text{ s.t. } j_{\ell} = 0.
		\end{cases}
	\end{equation}
	Note that this structure is homogeneous outside $(t_1,\dots,t_m-1)$, that is we have reduced the size of the inhomogeneous part. Repeated application of this principle along all dimensions $1 \leq \ell \leq d$ gives us that
	\begin{equation} \label{eq:movingaround}
		v_{\bi} = v_{\bi + \mathds{1}_{\ell}} \,\, \text{for all } \bi \in \NN_0^m \setminus \big( [0,t_1] \times \dots \times [0,t_{m-1}] \times \{0\} \big) \text{ and } 1 \leq \ell \leq m.
	\end{equation}
	Considering the relation given by $\bj = (2t_1,\dots,2t_m)$, which states that
	\begin{equation*}
		d \, v_{\bj} = v_{\bj \, \ominus \bc_1} + \dots + v_{\bj \, \ominus \bc_d} = u_{\bj} = 0.
	\end{equation*}
	Note that the choice of the constant $2$ was arbitrary, it just needs to be {\lq}large enough{\rq}. It follows that $v_{\bj} = 0$ and hence, again by relation~\eqref{eq:movingaround}, it follows that $v_{\bi} = 0$ for all $\bi \in \NN_0^m \setminus [0,t_1] \times \dots \times [0,t_m]$ as desired.
\end{proof}

\paragraph{Proof of Theorem~\ref{thm:main_result}.} Recall that $F_{\mA}(z) = f_{\mA}(z)^d$ and that the existence of a set $\mA$ for which $r_{\mA}(n;k_1,\dots,k_d)$ is a constant function for $n$ large enough would imply the existence of some polynomial $P(z) \in \ZZ[z]$ satisfying $P(1) \neq 0$ such that
\begin{equation*} \label{eq:rephrasedasgeneratingfunctions4}
	F_{\mA}(z^{k_1}) \cdots F_{\mA}(z^{k_d}) = \frac{P^d(z)}{(1-z)^d}.
\end{equation*}
Using Proposition~\ref{prop:FundamentalRelation} we see that if a such a function $F_{\mA}(z)$ were to exist, then the values $\{ r_{\bi} : \bi \in \NN_0^m \}$ together with $\bb_1, \dots, \bb_m$ and $\{ s_{\bj} : \bj \in \NN_0^m \setminus \{ \bold 0 \} \}$ would define an $m$--structure. By the requirements of the theorem we have $\bb_i \in \{0,1\}^m$ and since $k_1,\dots,k_d \geq 2$ we have $\bb_i \neq \mathbf{0}$. We may also assume that $S_{\ell} \neq \emptyset$ for all $1 \leq \ell \leq d$ as otherwise there exists some $\ell'$ such that $q_{\ell'} \mid k_i$ for all $1 \leq i \leq d$, in which case the representation function clearly cannot become constant, so that this $m$--structure would be regular. It would also be homogeneous outside some appropriate $\bt \in \NN_0^m$ as $P(z)$ is a polynomial and hence $s_{\bj} \neq 0$ only for finitely many $\bj \in \NN_0^m$. Finally, since $r_{\bi} \equiv -1 \mod d$ for all $\bi \in \NN_0^m$, this would contradict the statement of Lemma~\ref{lemma:all0}, proving Theorem~\ref{thm:main_result}. \qed

\section{Concluding Remarks}

We have shown that under very general conditions for the coefficients $k_1,\dots, k_d$ the representation function $r_{\mathcal{A}}(n; k_1,\dots,k_d)$ cannot be constant for $n$ sufficiently large. However, there are cases that our method does not cover. This includes those cases where at least one of the $k_i$ is equal to $1$. The first case that we are not able to study is the representation function $r_{\mathcal{A}}(n;1,1,2)$.

On the other side, let us point out that Moser's construction~\cite{Moser_1962} can be trivially generalized to the case where $k_i=k^{i-1}$ for some integer value $k \geq 2$. In view of our results and this construction, we state the following conjecture:
\begin{conjecture} There exists some infinite set of positive integers $\mathcal{A}$ such that $r_{\mathcal{A}}(n;k_1,\dots, k_d)$ is constant for $n$ large enough if and only if, up to permutation of the indices, $(k_1,\dots,k_d)=(1,k, k^2,\dots, k^{d-1})$, for some $k \geq 2$.
\end{conjecture}
The most likely candidates for a possible counterexample to this conjecture might be those where $(k_1,k_2,k_3)$ is either $(1,2,6)$ or $(1,2,8)$. One could possibly try to generalise Moser's approach to these scenarios, e.g. by using generalised bases. Understanding these cases would most likely indicate a path towards completely settling the question of S{\'a}rk\"ozy and S{\'o}s.

\bigskip

\noindent {\bf Acknowledgements. } We thank an anonymous referee for comments concerning the complex analytic aspects of our proof. We would also like to thank Oriol Serra for valuable input and fruitful discussions.

\bibliography{bib}
\bibliographystyle{abbrv}

\end{document}